\documentclass[preprint,12pt]{elsarticle}

\usepackage{amssymb}
\usepackage{amsmath}
\usepackage{amsfonts}
\usepackage{graphicx}
\usepackage{graphics}
\usepackage{tikz}
\usepackage{titlesec}

\newtheorem{theorem}{Theorem}

\newtheorem{corollary}[theorem]{Corollary}

\newtheorem{definition}[theorem]{Definition}
\newtheorem{example}[theorem]{Example}

\newtheorem{problem}[theorem]{Problem}
\newtheorem{proposition}[theorem]{Proposition}
\newtheorem{remark}[theorem]{Remark}

\def\qed{\vbox{\hrule
 \hbox{\vrule\hbox to 5pt{\vbox to 8pt{\vfil}\hfil}\vrule}\hrule}}

\usepackage{tikz}
\usepackage{tkz-graph}

\begin{document}

\journal{xxxxxxx}
\begin{frontmatter}

\title{On the spectra of $g$-circulant matrices and applications}
\title{On the spectra of $g$-circulant matrices and applications}

\author{Enide Andrade}
\ead{enide@ua.pt}
\address{CIDMA-Center for Research and Development in Mathematics and Applications
         Department of Mathematics, University of Aveiro, 3810-193, Aveiro, Portugal.}

\author{Luis Arrieta}
\address{Departamento de Matem\'{a}ticas, Universidad Cat\'{o}lica del Norte, Av. Angamos 0610 Antofagasta, Chile.}
\ead{luis.arrieta01@ucn.cl}

\author[]{Mar\'{\i}a Robbiano \corref{cor1}}
\cortext[cor1]{Corresponding author}
\address{Departamento de Matem\'{a}ticas, Universidad Cat\'{o}lica del Norte, Av. Angamos 0610 Antofagasta, Chile.}
\ead{mrobbiano@ucn.cl}

\begin{abstract} A $g$-circulant matrix of order $n$ is defined as a matrix of order $n$ where each row is a right cyclic shift in $g-$places to the preceding row. Using number theory, certain nonnegative $g$-circulant real matrices are constructed. In particular, it is shown that spectra with sufficient conditions so that it can be the spectrum of a real $g$-circulant matrix is not a spectrum with sufficient conditions so that it can be the spectrum of a real circulant matrix of the same order. The obtained results are applied to Nonnegative Inverse Eigenvalue Problem to construct nonnegative, $g$-circulant matrices with given appropriated spectrum. Moreover, nonnegative $g$-circulant matrices by blocks are also studied and in this case, their orders can be a multiple of a prime number.

\end{abstract}


\end{frontmatter}
\section{introduction}

\noindent A \textit{permutative} matrix is a square matrix where each row is a permutation of its first row. A \textit{circulant} matrix is a matrix where each row its a right cyclic shift in one place of the previous one. A \textit{$g-$circulant} matrix is a matrix where each row its a right cyclic shift in $g-$places of the previous row. Circulant matrices and $g$-circulant matrices are permutative.
An $n$-tuple of complex numbers,
\begin{equation}
\Sigma =\left( \lambda _{1},\lambda _{2},\ldots ,\lambda _{n}\right)
\label{list}
\end{equation}
is said to be \textit{realizable} by a nonnegative matrix $A$ of order $n$ if its components are the eigenvalues of $A.$ The \textit{Nonnegative Inverse Eigenvalue Problem} (NIEP) is a problem to determine necessary and sufficient conditions for a list of $n$ complex numbers to be realizable  by a nonnegative square matrix $A$ of order $n$.
If the list $\Sigma$ is realizable by a nonnegative matrix $A$, then we say that $A$ realizes $\Sigma$ or it is a \textit{realizing} matrix for $\Sigma$. Some results can be seen in \cite{Laffey1,Laffey2,Laffey-Smigoc}. This very difficult problem attracted the attention of many authors over the last $50$ years, and it was firstly considered by Sule\u{\i}manova in $1949$ (see \cite{Suleimanova}). Some partial results were obtained but it is still an open problem for $n \geq 5$.
In \cite{Loewy} the problem was solved for $n=3$ and for matrices of order $n=4$ the solution can be found in \cite{Mayo} and \cite{Meehan}.
In its general form it was studied in \cite{Boyle2,Guo3, Laffey,Laffey5,Laffey7,Loewy}. When the non-negative realizing matrix $A$ is required to be symmetric  the problem is designated \textit{Symmetric Nonnegative Inverse Eigenvalue Problem} (SNIEP) and it is also an open problem. It has also been a problem that had called much attention, see for instance, \cite{Fiedler,Laffey,LMc,Soules}. Another variant of this problem is to find lists of $n$ real numbers that can be lists of eigenvalues of nonnegative matrices of order $n$ and it is called the \textit{Real Nonnegative Inverse Eigenvalue Problem} (denoted by RNIEP).
Some results can be seen, for instance, in \cite{Laffey}. The structured NIEP is an analogous problem to NIEP with the difference that the matrix that realizes the list is structured. For instance, the matrix can be symmetric, Toeplitz,
Hankel, circulant, normal, permutative, etc., see  \cite{ AMRH, Fiedler,Laffey, LMc, MAR,PP} and the references cited therein.

\noindent In this paper we deal with structured matrices, in particular, permutative, circulant and $g$-circulant matrices. To study the spectrum with sufficient conditions in order that it can be the spectrum of a nonnegative permutative matrix allows us to answer the question if there exist sufficient conditions so that this spectrum can be realizable by a type of permutative matrix and at same time realizable by another type of permutative matrix.
The paper is organized as follows: In Section 2 we introduce some known results about $g-$circulant matrices and we prove some others; in Section 3 the spectrum of certain $g-$circulant matrices and the answer to the previous question is presented. It is also studied certain cases of $g-$circulant matrices that can be reconstructed knowing all its diagonal elements. Additionally, the case when the $g$-circulant matrix is reconstructed from its Perron root and its $n-1$ diagonal elements is also studied. In Section 4 sufficient conditions in order that a given $g-$circulant spectrum can be taken as the spectrum of a nonnegative $g-$circulant matrix are given. Nonnegative $g-$circulant matrices with given spectrum are obtained.
 Moreover, $g-$circulant matrices by blocks are defined and their spectrum are studied. Sufficient conditions in order that a given set can be taken as the spectrum of a certain class of $g-$circulant matrices by blocks are studied. Nonnegative $g$-circulant matrices by blocks with given spectrum are obtained.

\noindent The following notation will be used. A square nonnegative matrix $A$ is denoted by $A \geq 0$. $\Sigma(A)$ denotes the set of eigenvalues of a square matrix $A$. The transpose of $A$ is denoted by $A^{T}$. The trace of $A$ is $\rm{tr}(A).$ The set of invertible elements in the set of the residue classes congruence $(\rm{mod}\, n)$ will be denoted by  $U(\mathbb{Z} /n\mathbb{Z}).$ The set $S_n$ represents the symmetric group of order $n$. For integers $a,b$, $\rm{mdc}(a,b)$ denotes the greatest common divisor of $a$ and $b$. The identity matrix of order $n$ will be denoted by $I_{n}$ or just $I$ if its order can be deduced by the context.

\section{$g$-circulant matrices and some properties}
\noindent Circulant matrices are an important class of matrices and have many connections to  physics, probability and statistics, image processing, numerical analysis, number theory and geometry. Many properties of these matrices can be found in, for instance, \cite{Davies}. In terms of notation, these matrices are perfectly identified by its first row and we just write $C= C(\mathbf{c})=circ(c_{1},\ldots, c_{n}).$
 \begin{remark}
 \label{conju}
 If the vector $\mathbf{c}$ is real, and if $\mathbf{\lambda}=(\lambda_1,\lambda_2,\ldots,\lambda_n)$ is the vector of eigenvalues of $C,$ it is possible to see (\cite{rojo-soto}) that if $n$ is even then $\lambda_{1}$ is a real number and
$$\lambda_{n+1-k}=\overline{\lambda}_{k+1}, \quad 1 \leq k \leq \frac{n}{2}.$$ We note that, in this case, $\lambda_{\frac{n}{2}+1}$ is a real number.
 Analogously if $n$ is odd then,
 $$\lambda_{n+1-k}=\overline{\lambda}_{k+1}, \quad 1 \leq k \leq \frac{n-1}{2},$$ holds.
 \end{remark}

\noindent In this section a generalization of circulant matrices is recalled, namely $g$-circulant matrices. Some known properties of these interesting matrices are pointed out and others are shown.

\begin{definition} \cite{Davies}
A $g-$circulant matrix of order $n$ or simply $g$-circulant is a matrix in the following form:
\begin{eqnarray}
\label{mgcirc}
A &=&g-circ\left( a_{1},a_{2},\ldots ,a_{n}\right)  \\
&=&%
\begin{pmatrix}
a_{1} & a_{2} & a_{3} & \cdots  & a_{n} \\
a_{n-g+1} & a_{n-g+2} & a_{n-g+3} &  & a_{n-g} \\
a_{n-2g+1} & a_{n-2g+2} & a_{n-2g+3} &  & a_{n-2g} \\
\vdots  &  & \cdots  & \ddots  & \vdots  \\
a_{g+1} & a_{g+2} &  & \cdots  & a_{g}%
\end{pmatrix}.%
\end{eqnarray}
\end{definition}
Here the subscripts are taken $(\rm{mod} \, n)$, as for the circulant matrices.

\begin{remark} \cite{Davies}
\begin{itemize}
    \item If $0<g<n$, each row of $A$ is a right cyclic shift in $g$-places (or it is a left cyclic shift in $(n-g)$-places) to the preceding row.
    \item If $g >n$ a cyclic shift in $g$-places is the same cyclic shift in
    $g\, (\,\rm{ mod} \, n)$-places.
    \item By convention if $g$ is negative a right cyclic shift in $g$-places is equivalent to a left cyclic shift in $\left( -g\right) $-places. In consequence, for any integers $g$ and $g^{\prime},$ if $g \equiv g^{\prime } (\rm{mod} \, n)$ then a $g$-circulant and a $g^{\prime }$-circulant with the same first row are equal.
    \item If $g=n-1$, a $(-1)-$circulant matrix is obtained.
\end{itemize}
\end{remark}
Many examples of $g$-circulant matrices can be also found, (see \cite{Davies}).
The entries of a $g-$circulant matrix can also be specified. Let $A=\left( a_{ij}\right) _{1\leq i,j\leq n}$, from \cite{Davies}, $A$ is a $g-$circulant matrix if and only if
\begin{equation*}
a_{ij}=a_{i+1},_{j+g},\quad 1\leq i,j\leq n,
\end{equation*}
where the subscripts are taken (\rm{mod} n).

\noindent In an equivalent way, if $A=\left( a_{ij}\right) $ is a $g-circ\left(
a_{1},a_{2},\ldots ,a_{n}\right) $ then
\begin{equation*}
a_{ij}=a_{j-\left( i-1\right) g},\quad 1\leq i,j\leq n.
\end{equation*}

\noindent It results from the expression of the entries of a $g$-circulant matrix that if  $n$ is a prime number and $1<g< n$, a  $g$-circulant matrix of order $n$, $A$, is symmetric if and only if $g=n-1.$
In fact, if  $A=(a_{ij})=g-circ(a_1,\ldots, a_n)$ then $a_{ij}=a_{j-( i-1) g}, \ 1\leq i,j \leq n.$ Thus, $A$ is symmetric if and only if
$a_{j-( i-1)g} = a_{i-( j-1)g}, \  1\leq i,j \leq n,$ which is equivalent to $  j-( i-1) g \equiv i-( j-1) g \, (\rm{mod}\,n)$ and therefore $(j-i)(g+1) \equiv 0 \, (\rm{mod}\, n)$ which gives the equality $g\equiv n-1$.

\noindent Let $g>0.$ From  \cite{Davies}, the $g-$circulant matrices of order $n$ can be partitioned into two types, accordingly either $\rm{mdc}( n,g) =1$ or $\rm{mdc}( n,g) >1$. It is clear that all the rows of a $g-$circulant matrix are distinct if and only if $\rm{mdc}( n,g) =1$.
In this case, the rows of the $g$-circulant matrix can be permuted in such a way to re-obtain a classical circulant matrix. The same reasoning can be done for columns.


%

\noindent Here, if $\rm{mdc}(n,g) =1,$ the unique solution of the equation  $g x \equiv 1 \, ( \rm{mod} \,n)$ will be designated as $g^{-1}$, \cite{fraleigh}. Note that, from \cite{Davies}, if  $A$ is a non singular $g-$circulant matrix then $A^{-1}$ is a $g^{-1}-$circulant matrix.
\\
\noindent Let $Q_{g}=g-circ\left( 1,0,\ldots ,0\right) .$ Note that $Q_{g}$ is a permutation matrix and is unitary if and only if $\rm{mdc} ( n,g) =1.$ The next result and definition can be seen in  \cite{Davies}.



\begin{proposition}
\cite{Davies}
\label{qc}
The matrix $A$ is $g-$circulant if and only if $A$ is of the form $Q_{g}C$ where $C$ is a circulant matrix whose first row coincide with the first row of $A$.
\end{proposition}

\begin{definition}
\cite{Davies}
A $P D-$matrix of order $n$ is a matrix of the form
\begin{equation}
    S=P_\nu D \label{pd}
\end{equation}
where $P_{\nu}$ is a permutation matrix associated to $\nu \in S_{n}$ and $D$ is a diagonal matrix $$D=diag(d_1, d_2, \ldots, d_n).$$
\end{definition}
\noindent The $P D-$matrices are also called monomials.
\begin{remark}\label{primitive}
\noindent In \cite{Davies} it was shown that, if $S$ is as in (\ref{pd}), it is possible to calculate its characteristic polynomial from the decomposition of the matrix $P_{\nu}.$ In fact, let
$P_{\nu} = R^{\ast} diag (P_{\theta_{1}}, P_{\theta_{2}}, \ldots, P_{\theta_{m}})R,$ where $R$ is a permutation matrix, $\nu$ is decomposed into $m$ disjoint cycles of order $\theta_{j}$ and $P_{\theta_{j}}$ is the permutation matrix of order $\theta_{j}, \ 1\leq j \leq m$.
The eigenvalues of $S=P_{\nu}D$ are the totally of the $\theta_{j}-$th complex roots
\begin{equation*}
    (\tilde{d}_{j1}\ldots\tilde{d}_{j\theta_{j}})^{\frac{1}{\theta_{j}}}, \quad 1\leq j \leq m,
\end{equation*}
where
\begin{equation*}
    RDR^{\ast} = diag(\tilde{d}_{1}, \ldots, \tilde{d}_{n}).
\end{equation*}
\end{remark}

\noindent Let $F=(f_{ij}) = \frac{1}{\sqrt{n}}\left(
\omega ^{\left( i-1\right) \left( j-1\right) }\right) _{1\leq i,j\leq n}$, be the Discrete Fourier Transform matrix, where
\begin{equation} \label{omega}
\omega= \cos{\frac{2\pi}{n}}+ i \sin{\frac{2\pi}{n}}.
\end{equation}

\begin{proposition}
\label{espPD}
Let $n$ be a positive integer such that $\rm{mdc}( n,g) =1$. Then, the square matrix $A$ of order $n$, is a $g$-circulant matrix if and only $A=FQ_{g^{-1}}D F^{\ast} $ where $D$ is a diagonal matrix and $F$ is the Discrete Fourier Transform matrix.
\end{proposition}

\begin{proof}
Recall that that $F$ is unitary. Therefore, we need to prove that
\begin{equation*}
F^{\ast}A=Q_{g^{-1}}DF^{\ast}.
\end{equation*}.

\noindent Let $Q_{g}=(q_{ij}).$ Then
\begin{eqnarray*}
 q _{ij}&=&\left\{
\begin{tabular}{cc}
$1$ & if $j=\left( 1+\left( i-1\right) g\right) (\rm{mod} \, n)$
\\
$0$ & {\small other cases,}
\end{tabular}%
\right. \\ & \\ &=&\left\{
\begin{tabular}{cc}
$1$ & si $i=\left( 1+\left( j-1\right) g^{-1}\right) (\rm{mod} \, n) $ \\
$0$ & {\small other cases.}%
\end{tabular}%
\right.
\end{eqnarray*}

\noindent If $A$ is $g$-circulant, by Proposition \ref{qc}, $A=Q_{g}C$ where $C$ is a circulant matrix. As
$C=FDF^{\ast }$, and $A=Q_{g}FDF^{\ast }$ if $FQ_{g}=Q_{g^{-1}}F$ then, changing $g^{-1}$ by $g$ we can conclude that $A=FQ_{g^{-1}}DF^{\ast},$
which proves the result.
\noindent The entry $ij$ from the matrix on the left side of the equality is:
\begin{eqnarray*}
\left( FQ_{g}\right) _{ij} &=&\sum_{k=1}^{n}f_{ik} q _{kj}
\\
&=&f_{i\left( \left( 1+\left( j-1\right) g^{-1}\right) \right) } \\
&=&\frac{1}{\sqrt{n}}\left( \omega ^{\left( i-1\right) \left( j-1\right)
g^{-1}}\right) .
\end{eqnarray*}

\noindent On the other hand, the entry  $ij$ on the right side of the equality is:

\begin{eqnarray*}
\left( Q_{g^{-1}}F\right) _{ij} &=&\sum_{k=1}^{n}\left( Q_{g^{-1}}\right)
_{ik}f_{kj} \\
&=&f_{\left( 1+\left( i-1\right) g^{-1}\right) j} \\
&=&\frac{1}{\sqrt{n}}\left( \omega ^{\left( i-1\right) \left( j-1\right)
g^{-1}}\right) .
\end{eqnarray*}
which proves the result.
\end{proof}

\begin{example}
Consider the $2$-circulant matrices, $A$, $Q_{2}$ and the circulant matrix $C$ as follows:
$$ A=2-circ\left( 1,2,3,4,5\right), Q_{2}=2-circ\left( 1,0,0,0,0\right) \text{\, and\, } C=circ\left( 1,2,3,4,5\right).$$

\noindent The next computations are obtained with $4$ decimal places. By Corollary \ref{qc}
we have
\[
A=Q_{2}C.
\]
The spectrum of $A$ is:
\[
\Sigma \left( A\right) =\left\{ 15,\pm 3.3437,\pm 3.3437i\right\},
\]
and the spectrum of $C$ is:
\[
\Sigma \left( C\right) =\left\{ 15,-2.5\pm 3.441i,-2.5\pm 0.8123i\right\} .
\]
\noindent As in $U(\mathbb{Z}/5\mathbb{Z})$ the muitiplicative inverse of $2$ is $3$
we have
\[
Q_{2}^{-1}=Q_{3}.
\]

\noindent Then, by Proposition \ref{espPD} the $PD$-matrix,
\[
S=Q_{3}diag\left(
15,-2.5+3.441i,-2.5-3.441i,-2.5+0,8123i,-2.5-0,8123i\right)
\]
has the same eigenvalues of  $A$.
\end{example}

\section{Characterizing the spectra of certain $g$-circulant matrices}

\noindent In this section is determined in an explicit way the spectrum of a subclass of real $g$-circulant matrices. The next definition, is crucial here.
\begin{definition} \cite{Zhang}
A permutation matrix of order $n$ is called \textit{primary} if and only if the associated permutation is a cycle of length $n$.
\end{definition}

\noindent The next result is from \cite[Theorem 5.18]{Zhang} is important and also recalled:

\begin{theorem} \cite{Zhang} \label{Z}
A permutation matrix is irreducible if and only if it is similar by permutation matrices to a primary permutation matrix.
\end{theorem}

\begin{theorem}
Suppose that $\rm {mdc}( n,g) =1.$ Let $\mathbf{W}$ be a
submatrix of order $n-1$ such that
\begin{equation} \label{matrixQ}
Q_{g}=%
\begin{pmatrix}
1 & \mathbf{0} \\
\mathbf{0} & \mathbf{W}%
\end{pmatrix},
\end{equation}
then $\ \mathbf{W}$ is a primary permutation matrix if and only if
\begin{equation}
n\nmid \left( g^{\ell _{2}}-g^{\ell _{1}}\right), \quad 0\leq \ell
_{1}<\ell _{2}\leq n-2.  \label{cond2}
\end{equation}
\end{theorem}

\begin{proof}
Firstly it easy to see that the permutation $\nu $ that defines the matrix $Q_{g}$ is given by:
\begin{equation*}
\nu =
\begin{pmatrix}
1 & 2 & 3 & 4 & 5 & \ldots  &  & n-1 & n \\
1 & 1+g & 1+2g & 1+3g & 1+4g & \ldots  &  & 1+\left( n-2\right) g & 1+\left(
n-1\right) g
\end{pmatrix}.
\end{equation*}
or, in another way, if\ $k\neq 1,\ $
\begin{equation*}
k\rightarrow 1+\left( k-1\right) g, \quad 1\leq k\leq n.
\end{equation*}
\noindent Thus, the submatrix $\mathbf{W}$ is the matrix associated to the following permutation:
\begin{equation*}
2=1+g^{0}\rightarrow 1+g\rightarrow 1+g^{2}\rightarrow 1+g^{3}\rightarrow
\cdots \rightarrow 1+g^{n-2}.
\end{equation*}%
Therefore $\nu \ $ is a permutation with only one cycle of length $n-1$ if for
 $
\ell _{1}\neq \ell _{2}$,

\begin{equation*}
1+g^{\ell _{1}}\neq 1+g^{\ell _{2}}\left( \rm{mod} \, n\right), \quad%
0\leq \ell _{1}<\ell _{2}\leq n-2,
\end{equation*}%
that is equivalent to the condition in (\ref{cond2}).
\end{proof}




\begin{corollary}
Let $n$ be a prime integer and suppose that $\rm {mdc}( n,g) =1.$ Let $\mathbf{W}$ be a submatrix of order $n-1$ as in (\ref{matrixQ}).
Then $\ \mathbf{W}$ is a primary permutation matrix if and only if
\begin{equation}
g^{\ell }\text{ }\neq 1\ \left( \rm{mod} \ n\right), \quad 1\leq \ell
\leq n-2.  \label{cond4}
\end{equation}%
\end{corollary}

\begin{proof}
From (\ref{cond2}) we have
\begin{equation*}
g^{\ell _{1}}\text{ }\neq g^{\ell _{2}}\ \left( \rm{mod} \ n\right),  \quad 0\leq \ell _{1}<\ell _{2}\leq n-2.
\end{equation*}%
As $g$ is invertible in $\mathbb{Z}/n\mathbb{Z},$ the element $g^{\ell _{1}}$ is also invertible, therefore the previous equation is equivalent to
\begin{equation*}
g^{\ell _{2}-\ell _{1}}\text{ }\neq 1\ \left( \rm{mod}\ n\right), \quad 0\leq \ell _{1}<\ell _{2}\leq n-2,
\end{equation*}
which is clearly equivalent to the condition in
(\ref{cond4}).
\end{proof}

\begin{example}
\label{ejemplomodelo}
Consider that $n=11$ and $g=7.$ Therefore, the residue classes modulus $11$ of the powers of $7$ with exponent between $1$ and $9 (=n-2)$  are given in the table:

\begin{equation*}
\begin{tabular}{|c|c|c|c|c|c|c|c|c|}
\hline
$7$ & $7^{2}$ & $7^{3}$ & $7^{4}$ & $7^{5}$ & $7^{6}$ & $7^{7}$ & $7^{8}$& $7^{9}$ \\
\hline
$7$ & $5$ & $2$ & $3$ & $10$ & $4$ & $6$ & $9$ & $8$ \\ \hline
\end{tabular}%
\end{equation*}%
Thus, the condition (\ref{cond4}) holds and then $\mathbf{W}$ is primary. In fact, let \begin{equation*}
Q_{7}=%
\begin{pmatrix}
1 & \mathbf{0} \\
\mathbf{0} & \mathbf{W}
\end{pmatrix},
\end{equation*}
where
\begin{equation*}
\mathbf{W=}%
\begin{pmatrix}
0 & 0 & 0 & 0 & 0 & 0 & 1 & 0 & 0 & 0 \\
0 & 0 & 1 & 0 & 0 & 0 & 0 & 0 & 0 & 0 \\
0 & 0 & 0 & 0 & 0 & 0 & 0 & 0 & 0 & 1 \\
0 & 0 & 0 & 0 & 0 & 1 & 0 & 0 & 0 & 0 \\
0 & 1 & 0 & 0 & 0 & 0 & 0 & 0 & 0 & 0 \\
0 & 0 & 0 & 0 & 0 & 0 & 0 & 0 & 1 & 0 \\
0 & 0 & 0 & 0 & 1 & 0 & 0 & 0 & 0 & 0 \\
1 & 0 & 0 & 0 & 0 & 0 & 0 & 0 & 0 & 0 \\
0 & 0 & 0 & 0 & 0 & 0 & 0 & 1 & 0 & 0 \\
0 & 0 & 0 & 1 & 0 & 0 & 0 & 0 & 0 & 0%
\end{pmatrix}.
\end{equation*}
The eigenvalues of $\mathbf{W}$ with $4$ decimal places are in the set:%
\begin{equation*}
\tiny{\Sigma \left( \mathbf{W}\right)} \tiny{=\left\{ \pm 1.0000,-0.8090\pm
0.5878i,-0.3090\pm 0.9511i,0.3090\pm 0.9511i,0.8090\pm 0.5878i\right\}.}
\end{equation*}
\end{example}




\begin{corollary}
\label{equivallente}
Let $n$ and $g$ integers such that $n$ is prime and $\rm{mdc}( n,g) =1.$ Let $\mathbf{W}$ be a submatrix of order $n-1$ as in (\ref{matrixQ}).
Then $\mathbf{W}$ is a primary permutation matrix if and only if the cyclic subgroup spanned by $g$ in $U(\mathbb{Z}/n\mathbb{Z})$ has order $n-1$, that is, if and only if $g$ is a cyclic generator of $U(\mathbb{Z}/n\mathbb{Z}).$
\end{corollary}
\begin{proof} The matrix $ \mathbf{W}$ is a primary permutation matrix if and only if the condition
(\ref{cond4}) holds and therefore, if and only if, the powers in $U(\mathbb{Z}/n\mathbb{Z})$
\begin{equation}
g^{\ell }, \quad  0\leq \ell \leq n-2.  \label{cond8}
\end{equation}
are pairwise distinct. This leads to the fact that the condition in the statement must hold.
\end{proof}

\begin{corollary}
\label{supercorolario}
Let $n$ and $g$ integers such that $n$ is prime, $g<n$ and $\rm{mdc}( n,g) =1.$
Let $\mathbf{W}$ be a submatrix of order $n-1$ as in (\ref{matrixQ}).
Then, $ \mathbf{W}$ is a primary permutation matrix if and only if \ $\forall 1\leq d<n-1,$
\begin{eqnarray}\label{supercondicion}
  d|(n-1) \  \Rightarrow \  g^{d}\neq 1(\rm{mod} \,n)
\end{eqnarray}

\end{corollary}
\begin{proof}
As the order of the set $U(\mathbb{Z}/n\mathbb{Z})$ is $n-1$, and the order of the cyclic subgroup generated by  $y\in U(\mathbb{Z}/n\mathbb{Z})$ is a divisor $d$ of $n-1$, \cite{fraleigh}, if the condition (\ref{supercondicion}) holds then, the order of the cyclic subgroup generated by $g$ must be $n-1.$ Thus, by Corollary \ref{equivallente} the matrix $ \mathbf{W}$ is a primary permutation matrix.
Reciprocally, if $ \mathbf{W}$ is a primary permutation matrix, by Corollary \ref{equivallente}
the cyclic subgroup generated by $g\in U(\mathbb{Z}/n\mathbb{Z})$ has order $n-1$ implying that  the statement holds.
\end{proof}
\begin{example}
At Example \ref{ejemplomodelo}, $n-1=10$ and the positive divisors of $10$, strictly less than $10$ are $1,2,5$. Recall that $g=7$ and now we have:
\begin{align*}
    7^{1}&=7\neq 1\pmod{\ 11},\\
    7^{2}&=5\neq 1\pmod{\ 11},\\
    7^{5}&=10\neq 1\pmod{\ 11}.
\end{align*}
Therefore $7$ is a cyclic generator of $U(\mathbb{Z}/11\mathbb{Z})$. Note that $3^{5}\equiv 1 \pmod{\ 11}$. Then $3$ is not a cyclic generator in $U(\mathbb{Z}/11\mathbb{Z})$.
\end{example}
\begin{theorem}
\label{eigenvalues}
Let $n$ and $g$ integers, such that $n$ is prime, $g<n$, $\rm{mdc}(n,g) =1$, $\mathbf{W}$ is a submatrix of $Q_g$ as in (\ref{matrixQ}), and $g$ is a cyclic generator of $U(\mathbb{Z}/n\mathbb{Z})$. Let $A
$ and $C$ real square submatrices of order $n$, where $A$ is $g$-circulant and $C$ is circulant, repectively, such that
\[
A=Q_{g}C\text{,}
\]
where $Q_{g}$ is as in (\ref{matrixQ}).
Then, the eigenvalues of $A$ are in the set
\begin{equation}\label{set}
\left\{ \lambda _{1},\beta ^{%
\frac{1}{n-1}},\beta ^{\frac{1}{n-1}}\varphi ,\beta ^{\frac{1}{n-1}}\varphi
^{2},\ldots ,\beta ^{\frac{1}{n-1}}\varphi ^{n-2}\right\},
\end{equation}
where
$$
\varphi
=\cos \frac{2\pi }{n-1}+i\sin \frac{2\pi }{n-1},$$
$$\beta =\left(
\left\vert \lambda _{2}\right\vert \left\vert \lambda _{3}\right\vert \ldots
\left\vert \lambda _{\frac{n-1}{2}}\right\vert \right) ^{2}, \ \text{and} \quad
\lambda_{1},\lambda _{2},\ldots ,\lambda _{\frac{n-1}{2}}$$ are the eigenvalues of the circulant matrix $C.$
\end{theorem}
\begin{proof}
As $\mathbf{W}$ is a primary matrix, it is invertible with primary inverse $\mathbf{W}^{-1}$. Moreover, the matrix $(Q_g)^{-1}$ is $$(Q_g)^{-1}=Q_{g^{-1}}=\left(\begin{matrix}
1&0\\
0 & \mathbf{W^{-1}}\end{matrix}\right).$$
From Remark \ref{primitive} and Proposition \ref{espPD} the eigenvalues are the $(n-1)$-th roots of the product of the last $n-1$ eigenvalues of $C$. Therefore, from Remark \ref{conju} for the case of real matrices $A$ and $C$ respectively, and $n$ odd, the eigenvalues are displayed into conjugated pairs and therefore the product of the $n-1$ eigenvalues is as presented at the statement of the theorem.
\end{proof}

\begin{corollary}
Let $n$ be a prime integer and $g_1$ and $g_2 $ two cyclic generators of $U(\mathbb{Z}/n\mathbb{Z}).$ Let $C$ be a circulant matrix of order $n$ and, consider the following two $g_1$-circulant and $g_2$-circulant matrices of order $n,$ respectively,  $A=Q_{g_{1}}C$ and $B=Q_{g_{2}}C.$ Then $A$ and $B$ have the same eigenvalues. In particular, if $g_2=(g_1)^{-1}$, the $g$-circulant matrices $Q_{g_{1}}C$ and $Q_{g_{1}^{-1}}C$ have the same eigenvalues.

 \end{corollary}

\begin{remark}
It is clear that the spectrum in (\ref{set}) doesn't have the sufficient conditions to be the spectrum of a circulant matrix as it doesn't have the conditions of Remark \ref{conju}, in the odd case.
\end{remark}
\section{Reconstructing certain $g$-circulant matrices from its diagonal entries}

\noindent In this section we show that for a certain $n$ and $g$, a $g$-circulant matrix its completely determined by its diagonal entries.

\begin{remark}
\bigskip Note that the vector of the diagonal entries of the matrix in (\ref{mgcirc}) is the  vector
\begin{equation}
\Upsilon \left( A\right) =\left( a_{1},a_{n-g+2},a_{n-2g+3},\ldots
,a_{n-\left( \ell -1\right) g+\ell },\ldots ,a_{g}\right)^{T} .  \label{vecdiag}
\end{equation}%
Moreover,
\begin{equation*}
a_{1}=a_{n-0g+1}\ \text{\, and\, }\ a_{g}=a_{n-\left( n-1\right) g+n}.
\end{equation*}%
\end{remark}

\begin{theorem}
Let $n$ be a prime integer and let $g<n$ be a cyclic generator of $U(\mathbb{Z}/n\mathbb{Z}).$ Let
\begin{equation}
\label{vecfil}
\Phi \left( A\right) =\left( a_{1},a_{2},\ldots ,a_{n}\right)^{T}
\end{equation}
be the vector corresponding to the first row of a $g$-circulant matrix $A.$
Then
 \begin{equation}
\Upsilon \left( A\right) =Q_{(n-g+1) }\Phi \left( A\right) \label{produc}
\end{equation}.
\end{theorem}
\begin{proof} For the subscripts of the diagonal elements of $A$, we note that
\begin{equation*}
\left( n-\left( k-1\right) g+k\right) \equiv \left( n-\left( \ell -1\right)
g+\ell \right) \left( \rm{mod} \ n\right), \quad 1\leq k<\ell \leq n
\end{equation*}%
if and only if
\begin{equation*}
n\mid (\ell -k) \text{ or\ }n\mid (g-1).
\end{equation*}
The former condition does not hold because $\ell -k\leq n-1<n$ and the latter condition also does not hold as well because $g$ is strictly less than $n$. Then, the diagonal entries of the $g$-circulant matrix $A$ in (\ref{mgcirc}) are a permutation of the first row of $A.$ Then, the difference $(\rm{mod} \, n)$  between two consecutive subscripts of the vector in (\ref{vecdiag}) is $1-g$, and  (\ref{produc}) holds.
\end{proof}
\begin{remark}
From the previous theorem we can say that for a given $n$ prime and $g$ a cyclic generator of $U(\mathbb{Z}/n\mathbb{Z})$  and, for a vector with $n$ given numbers $\mathbf{b}=\left(b_{1},b_{2,}\ldots ,b_{n}\right),$ a $g$-circulant matrix $A$ can be constructed in such a way that its diagonal elements are the elements of  $\mathbf{b}$.
Additionally, if in  $\mathbf{b}$, one of the coordinates is unknown but the greatest eigenvalue of $A$, say $\beta _{1},$ is known then, a $g$-circulant matrix $A$ whose diagonal entries are the elements of $\mathbf{b}$ and whose greatest eigenvalue is $\beta _{1}$ can also be constructed. This last fact is due to the conditions for $n$ and $g$, the $\rm{tr}(A)$ coincide with $\beta_1$ which allows to obtain the unknown coordinate of $\mathbf{b}.$
\end{remark}

\begin{example}
Consider the matrix $A=3-circ(1,2,3,4,5,6,7).$
whose diagonal is
\begin{equation*}
\mathbf{b}=\Upsilon \left( A\right) =\left( 1,6,4,2,7,5,3\right)^{T} \text{,\ }
\end{equation*}%
moreover
\begin{equation*}
Q_{1+\left( 7-3\right) }=Q_{5}=%
\begin{pmatrix}
1 & 0 & 0 & 0 & 0 & 0 & 0 \\
0 & 0 & 0 & 0 & 0 & 1 & 0 \\
0 & 0 & 0 & 1 & 0 & 0 & 0 \\
0 & 1 & 0 & 0 & 0 & 0 & 0 \\
0 & 0 & 0 & 0 & 0 & 0 & 1 \\
0 & 0 & 0 & 0 & 1 & 0 & 0 \\
0 & 0 & 1 & 0 & 0 & 0 & 0%
\end{pmatrix}.
\end{equation*}%
Then
\begin{equation*}
Q_{5}\Phi \left( A\right)=\left( 1,6,4,2,7,5,3\right)^{T}
=\Upsilon \left( A\right),
\end{equation*}
which verifies the equality in (\ref{produc}).
We can conclude, in this case, that given the diagonal elements, the first row of the $g$-circulant matrix (and in consequence all the remaining entries of the $g$-circulant matrix) can be obtained using the relation in (\ref{produc}) as this relation means that
\begin{equation*}
\Phi \left( A\right) =\left( Q_{5}\right) ^{-1}\Upsilon \left( A\right)
.
\end{equation*}
\end{example}

\section{An inverse eigenvalue problem for $g$-circulant matrices}

\noindent The spectra of certain classes of permutative matrices were studied in \cite{MAR2, MAR, robbiano}. In particular, spectral results for matrices partitioned into symmetric blocks of order $2$ were given. Using those results, sufficient conditions in order that a given list can be taken as the list of eigenvalues of a permutative matrix were obtained and the corresponding permutative matrices were constructed. In the light of these results in this section, we present sufficient conditions in order that a list $\Lambda$ can be taken as the spectrum of a class of $g$-circulant matrices that are a class of permutative matrices. Moreover, results on real circulant, real $g$-circulant and matrices partitioned into blocks are given and some results obtained in \cite{AMRH} are reviewed. Additionally, sufficient conditions for certain lists to be realizable by a class of permutative matrices are shown.

\subsection{Some results revisited}

\noindent We recall here the definition of permutative and $\mathbf{\gamma}$\emph{-permutative} matrices presented in \cite{MAR} and \cite{PP}, respectively. Moreover, some useful results used in the sequel will be also presented.

\begin{definition} \cite{PP}
A square matrix of order $n$ with $n\geq 2$ is called a \textit{permutative matrix} or permutative when all its rows (up to the first one) are permutations of its first row.
\end{definition}

\noindent The next definition was introduced in \cite{MAR}.

\begin{definition}  \cite{MAR} \label{ept}
Let $\mathbf{\gamma }=\left( \gamma _{1},\ldots ,\gamma _{n}\right) $ be an $n$-tuple whose elements are permutations in the symmetric group\ $S_{n}$, with $\gamma _{1}=id$.\ Let $\mathbf{a=}\left( a_{1},\ldots ,a_{n}\right) \in
\mathbb{C}^{n}$. Consider the vector,
\begin{equation*}
\gamma _{j}\left( \mathbf{a}\right) =\left( a_{\gamma _{j}\left( 1\right)
},\ldots ,a_{\gamma _{j}\left( n\right) }\right)
\end{equation*}%
and the matrix

\begin{equation}
\gamma \left( \mathbf{a}\right) =%
\begin{pmatrix}
\gamma _{1}\left( \mathbf{a}\right),
\gamma _{2}\left( \mathbf{a}\right),
\ldots,
\gamma_{n-1}\left( \mathbf{a}\right),
\gamma _{n}\left( \mathbf{a}\right)
\end{pmatrix}^{T}
.  \label{permut}
\end{equation}
A matrix $A$ of order $n$, is $\mathbf{\gamma}$\emph{-permutative} if
$A=\gamma \left( \mathbf{a}\right) $ for some $n$-tuple $\mathbf{a}$.
\end{definition}

\begin{definition}
\cite[Definition 8]{MAR}
If $A$ and $B$ are $\mathbf{\gamma }$-permutative by a common vector
$\mathbf{\gamma }=\left( \gamma _{1},\ldots ,\gamma _{n}\right)$
then they are called \textit{permutatively equivalent}.
\end{definition}

\noindent The class of circulant matrices and its properties were described in \cite{Karner}.

\noindent Let $p$ be a prime interger and $\mathbf{a}=\left( a_{1},a_{2},\ldots ,a_{p}\right)$ a $p$-tuple of complex numbers. We recall here the definition.

\begin{definition} \cite{Karner}
A \emph{\ real circulant matrix} is a matrix of the form
\begin{equation*}
circ\left( \mathbf{a}\right) =
\begin{pmatrix}
a_{1} & a_{2} & \ldots  & \ldots & a_{p} \\
a_{p} & a_{1} & a_{2} & \ldots & a_{p-1} \\
a_{p-1} & \ddots  & \ddots  & \ddots  & \vdots  \\
\vdots  & \ddots  & \ddots  & a_{1} & a_{2} \\
a_{2} & \ldots  & a_{p-1} & a_{p} & a_{1}
\end{pmatrix}.
\end{equation*}
\end{definition}



\noindent The spectrum of circulant matrices was also characterized in \cite{Karner}.
In fact, for
$\mathbf{c}=(c_1,\ldots,c_{p})$ and $C(\mathbf{c})=circ(c_1,\ldots,c_{p})$ then, $C\left( \mathbf{c}\right) =FD \left( \mathbf{c}\right) F^{\ast },$ with
$$
D\left( \mathbf{c}\right) =diag\left( \lambda _{1}\left( \mathbf{c}\right) ,\lambda
_{2}\left( \mathbf{c}\right) ,\ldots ,\lambda _{p}\left( \mathbf{c}\right) \right), $$
and
\begin{eqnarray}
\label{eigenvalues2}
\text{\ }\lambda _{k}\left( \mathbf{c}\right) =\sum\limits_{\ell=1}^{p}c_{\ell}\tau
^{(k-1)(\ell-1)}\text{,\quad\ }1\leq k \leq p,
\end{eqnarray}
where the entries of the discrete Fourier transform
$F=\left(
f_{k\ell}\right)$ are:
\begin{equation}
\label{fourier-m}
f_{k\ell}=\frac{1}{\sqrt{p}}\tau ^{(k-1)(\ell-1)},\quad 1\leq k,\ell\leq p\ ,
\end{equation}%
with
\begin{equation}
\label{omega_root}
\tau =\cos \frac{2\pi }{p}+i\sin \frac{2\pi }{p}.
\end{equation}.

\noindent Note that, by Remark \ref{conju}, $c_1,... ,c_{p}$ are real numbers if and only if $\lambda_1\in \mathbb{R}$ and $\lambda_ {n+1-k}=\overline {\lambda}_ {k+1}, \quad 1\leq k \leq p.$

\noindent Moreover, if $\mathbf{c}$ is defined as before and $\Lambda=\Lambda(\mathbf{c})=\left( \lambda _{1}\left( \mathbf{c}\right) ,\lambda _{2}\left(
\mathbf{c}\right) ,\ldots ,\lambda _{p}\left( \mathbf{c}\right) \right)$
then,
\begin{eqnarray}
\label{coefficients}
c_{k}=\frac{1}{p}\sum\limits_{\ell=1}^{p}\lambda _{\ell}\tau ^{-(k-1)(\ell-1)}\text{%
,\quad } 1 \leq k \leq p-1.
\end{eqnarray}

\noindent In \cite{AMRH}, using a result from \cite{Williamson} for matrices partitioned into blocks and into circulant blocks, the following spectral result was presented.
\begin{theorem}
\cite{AMRH}
\label{main}
Let $\mathbf{K}$ be an algebraically closed field of characteristic $0$ and suppose that $C=\left( C(i,j)\right) $ is an  $np \times np$ matrix partitioned into $p^2$ circulant blocks of order $n$ where for $ 1\leq i,j\leq p,$
\begin{equation}
\label{mtcsa}
C=\left( C(i,j)\right)_{1\leq i,j\leq p}, \ C(i,j)=circ\left(\mathbf{c}(i,j)\right),
\end{equation}
with
\begin{eqnarray*}
\mathbf{c}(i,j)=(c_1(i,j),\ldots,c_{n}(i,j)),
\\
c_k(i,j)\in \mathbf{K},\ 1\leq i,j\leq p, \quad  1\leq k\leq n.
\end{eqnarray*}
Then,
\begin{eqnarray}
\label{unionofsets}
\Sigma \left( C\right) =\bigcup_{k=1}^{n}\Sigma \left( S_k\right),
\end{eqnarray}
where, if $\omega$ is as in (\ref{omega}), then $\forall{1\leq k\leq n},$
\begin{align}
\label{mtcsk}
S_k=\left( s_k(i,j)\right)_{1\leq i,j\leq p}, \quad
s_k(i,j)=\sum\limits_{\ell=1}^{n}c_{\ell}\left(i,j\right)\omega
^{(k-1)(\ell-1)}.
\end{align}

\end{theorem}

\noindent The next result is a direct consequence of (\ref{coefficients}).
\begin{corollary}
\cite{AMRH}
\label{remark1}
Let $S_{\ell}$ be the matrix defined in (\ref{mtcsk}), $ 1 \leq \ell \leq n$. Consider $c(u,v):=(c_1(u,v),\ldots, c_{n}(u,v))^T$ and  $C(u,v)=circ(c(u,v))$. Then, the matrix $C$ in (\ref{mtcsa}) is nonnegative if and only if the matrix

\begin{eqnarray*}
\label{summatrix}
L_k:=\begin{pmatrix}
c_k(1,1) & c_k(1,2)&\ldots& c_k(1,p)\\
\vdots&\vdots&\ddots&\vdots\\
c_k(p,1) & c_k(p,2)&\ldots& c_k(p,p)
\end{pmatrix}
&=&
\frac{1}{n}\sum\limits_{\ell=1}^{n} S _{\ell}\omega ^{-(k-1)(\ell-1)},
\end{eqnarray*}

\noindent $1\leq k\leq n$, where $\omega$ was defined as in $(\ref{omega})$
is nonnegative and, in this case the matrix  $S_{1}$ is nonnegative.

\end{corollary}





\subsection{Construction of nonnegative $g$-circulant matrices given its spectrum
}

\noindent The following problem is considered:

\begin{problem}Let $\beta_{1}\in \mathbb{R}$ and $\beta_{2}>0$ real numbers. Let $p$ be an odd prime number and $g$ a cyclic generator of $U(\mathbb{Z}/p\mathbb{Z})$. Let
\begin{equation}
\varphi =\cos \frac{2\pi }{p-1}+i\sin \frac{2\pi }{p-1}.  \label{r1pr}
\end{equation}
Consider the list
\begin{equation}
\Sigma =\left( \beta _{1},\beta _{2},\beta _{2}\varphi ,\beta _{2}\varphi
^{2},\ldots ,\beta _{2}\varphi ^{p-2}\right) .  \label{list1}
\end{equation}
Find a real matrix $g$-circulant matrix $A$ whose set of eigenvalues are the components of  $\Sigma $.
\end{problem}

\noindent The answer to this problem is constructive and it is presented in the following result.

\begin{theorem}
\label{constructivo}
Let $p>2$ be a prime number. Suppose that  $g$ is a cyclic generator of $U(\mathbb{Z}/p\mathbb{Z}).$ Let $\varphi$ be the $p-1$ primitive root of the unit defined as in  (\ref{r1pr}).
Consider the list given in (\ref{list1})\ and let us define the auxiliary list
\begin{equation*}
\Sigma ^{\prime }=\left( \beta _{1},\beta _{2}\tau ,\beta _{2}\tau
^{2},\beta _{2}\tau ^{3},\ldots ,\beta _{2}\tau ^{p-1}\right)
\end{equation*}
where $\tau$ is defined in (\ref{omega_root}). Then if $C$ is the real circulant matrix whose eigenvalues are the components of $\Sigma ^{\prime }$ then $A=Q_{g}C$ is a real $g$-circulant matrix whose eigenvalues can be ordered as the components of $\Sigma $ in (\ref{list1}).
\end{theorem}

\bigskip

\begin{proof}
Firstly observe that if
\begin{equation}
\left( \lambda _{1},\lambda _{2},\ldots ,\lambda _{p}\right) :=\left( \beta
_{1},\beta _{2}\tau ,\beta _{2}\tau ^{2},\beta _{2}\tau ^{3},\ldots
,\beta _{2}\tau ^{p-1}\right) ,  \label{parconj}
\end{equation}
then
\begin{equation}\label{mancha}
\left\vert \lambda _{1}\right\vert \left\vert \lambda _{2}\right\vert \ldots
\left\vert \lambda _{p}\right\vert =\beta _{1}\beta _{2}^{p-1}\Rightarrow
\left\vert \lambda _{2}\right\vert \ldots \left\vert \lambda _{p}\right\vert
=\beta _{2}^{p-1}
\end{equation}
$\Rightarrow $%
\begin{equation*}
\left( \left\vert \lambda _{2}\right\vert \ldots \left\vert \lambda
_{p}\right\vert \right) ^{\frac{1}{p-1}}=\beta _{2}.
\end{equation*}
\noindent Let
\begin{equation*}
C=circ\left( a_{1},a_{2},\ldots ,a_{p}\right) \ \text{and} \ A=Q_{g}C.
\end{equation*}%

\noindent Then by Theorem \ref{eigenvalues} and (\ref{mancha}), the list of eigenvalues of $A$ is the one given in (\ref{list1}
). Note that the list that is used to define the circulant matrix $C$ in (\ref{parconj}) is a pair conjugate list as in \cite{rojo-soto} and therefore the matrices $C$ and $A$ are real matrices.
\end{proof}


\begin{theorem}
\label{suficient}
Let $p$ be a prime integer and consider $g$ a cyclic generator of $U(\mathbb{Z}/p\mathbb{Z}).$  
If
\begin{equation}
\label{condi4}
\beta _{1}\geq \beta _{2}\geq 0,
\end{equation}
then there exists a nonnegative $g$-circulant matrix $A$ of order $p$ such that its eigenvalues are those from the list presented in (\ref{list1}).
\end{theorem}

\begin{proof}
According to the previous proof and using the expressions in (\ref{coefficients}), the entries of the first row of the circulant matrix $C$
are
\begin{equation*}
c_{j}=\frac{1}{p}\left( \beta _{1}+\sum_{u=2}^{p}\beta _{2}\tau
^{\left( u-1\right) }\overline{\tau }^{\left( u-1\right) \left( j-1\right)
}\right), \quad  1\leq j\leq p, \ j\neq 2.
\end{equation*}
Since, for $j\neq 2, \sum_{u=2}^{p}\beta _{2}\tau
^{\left( u-1\right)}\overline{\tau }^{\left( u-1\right) \left( j-1\right)
}=-\beta_2$, the first part of the result follows. For the second part of the result we need note that $$pc_2=\beta_1+(p-1)\beta_2.$$ Therefore, the condition in (\ref{condi4}) holds for the matrix $C$ and then $A$ is non-negative. In consequence, $A$ is a  $g$-circulant nonnegative matrix with spectrum as in (\ref{list1}).
\end{proof}
\begin{remark}
By considering $\beta_1=\beta_2$ in Theorem \ref{condi4} the mentioned $g$-circulant nonnegative matrix $A$ becomes $g-circ(0,\frac{\beta_1+(p-1)\beta_2}{p},0,\ldots,0).$
\end{remark}
\begin{example}\
Consider the list with $7$ elements, $\left( 6,5,5\varphi ,5\varphi ^{2},5\varphi ^{3},5%
\overline{\varphi }^{2},5\overline{\varphi }\right) ,$ where
\begin{equation*}
\varphi =\cos \frac{2\pi }{6}+i\sin \frac{2\pi }{6}=\cos \frac{\pi }{3}%
+i\sin \frac{\pi }{3}.
\end{equation*} In this example it is determined a nonnegative, $3$-circulant matrix $A$ whose spectrum is the given list. In fact, consider the list $$\left( 6,5\tau
,5\tau ^{2},5\tau ^{3},5\overline{\tau }^{3},5\overline{\tau }^{2},5%
\overline{\tau }\right) $$ as the list of eigenvalues of a circulant matrix
$C$. With $\tau=\cos{\frac{2\pi}{7}}+i\sin{\frac{2\pi}{7}}$, and using MATLAB we have
\begin{equation*}
C=circ\left( 0.1429,5.1429,0.1429,0.1429,0.1429,0.1429,0.1429\right) .
\end{equation*}

\noindent Let $a=0.1429$, and $b=5.1429$.\ Again, using MATLAB we obtain the $3$-circulant matrix
\begin{equation*}
A=%
\begin{pmatrix}
a & b & a & a & a & a & a \\
a & a & a & a & b & a & a \\
b & a & a & a & a & a & a \\
a & a & a & b & a & a & a \\
a & a & a & a & a & a & b \\
a & a & b & a & a & a & a \\
a & a & a & a & a & b & a%
\end{pmatrix}%
\end{equation*}%
whose eigenvalues are
\begin{equation*}
\Sigma \left( A\right) =\left\{ 6,\pm 5,2.5\pm 4.33i,-2.5\pm 4.33i\right\} .
\end{equation*}
Note that only the greatest eigenvalue doesn't have modulus $5$ and $-5$
corresponds to the element $5\varphi ^{3}$ of the given list.
\end{example}

\subsection{$g$-ciculant matrices partitioned by blocks}

\begin{definition}
\cite{AMRH} A partitioned matrix into blocks is called \textit{permutative by blocks} when its row blocks are permutations of the first row block.
\end{definition}

\begin{theorem}
\cite{AMRH}
Let $C$ be the matrix partitioned into blocks defined in (\ref{mtcsa}). For $1\leq k \leq n,$ let $S_k$ be the class of matrices connected to $C$ as defined in (\ref{mtcsk}). The matrix $C$
is permutative by blocks if and only if the matrices $S_k$ are pairwise permutatively equivalent.
\end{theorem}

\begin{definition}
We say that the matrix partitioned into blocks  $A=\left( A_{ij}\right) _{1\leq i,j\leq p}$
where each block $A_{ij}$ is $n\times n$ and $p$ is an odd prime number, is a $g$-circulant matrix by blocks if
\begin{equation*}
A_{ij}=A_{i+1},_{j+g}, \quad 1\leq i,j\leq p,
\end{equation*}
where the subscripts are all considered congruent $(\rm{mod}\, p)$. In this case, all the blocks of $A$ depend only from its first row block. Considering its first row blocks as
$A_{1},A_{2},\ldots ,A_{p}$
we have
\begin{equation*}
A=%
\begin{pmatrix}
A_{1} & A_{2} & A_{3} &  & \ldots & A_{p} \\
A_{p-g+1} & A_{p-g+2} & A_{p-g+3} &  & \ldots & A_{p-g} \\
A_{p-2g+1} & A_{p-2g+2} & A_{p-2g+3} &  & \ldots & A_{p-2g} \\
\vdots & \vdots & \vdots & \ddots &  & \vdots \\
&  &  &  &  &  \\
A_{g+1} & A_{g+2} & A_{g+3} &  & \ldots & A_{g}%
\end{pmatrix}%
.
\end{equation*}%
We will denote $A$ as
\begin{equation*}
A=g-circ\left( A_{1},A_{2},A_{3},\ldots ,A_{p}\right) .
\end{equation*}%
Suppose that the first row of blocks of $A$ are
$p$ circulant matrices of order $n$ and let
\begin{equation*}
A_{k}=circ\left( a_{1}\left( k\right) ,a_{2}\left( k\right) ,\ldots
,a_{n}\left( k\right) \right),\quad  1\leq k\leq p.
\end{equation*}
By Theorem \ref{main} it was shown that

\begin{equation*}
\Sigma \left( A\right) =\bigcup_{k=1}^{n}\Sigma \left( S_{k}\right),
\end{equation*}%
where (see \cite{robbiano})
\begin{equation*}
S_{k}=g-circ\left( \lambda _{k}\left( A_{1}\right) ,\lambda _{k}\left(
A_{2}\right) ,\ldots ,\lambda _{k}\left( A_{p}\right) \right), \quad
1\leq k\leq n,
\end{equation*}

\noindent with $$\lambda_k(A_i)=\sum_{j=1}^{n}a_{j}(k)\omega^{(k-1)(j-1)}.$$
\end{definition}

\noindent We will give an answer to the following problem.
\begin{problem}
Let $p$ be a prime integer and consider $g$ a cyclic generator of $U(\mathbb{Z}/p\mathbb{Z})$. Let $\varphi$ as in (\ref{r1pr}). Given the set
\begin{equation}
\label{dado1}
\mathcal{B}=\bigcup_{k=1}^{n}\left\{ \beta _{1k},\beta _{2k},\beta _{2k}\varphi ,\beta
_{2k}\varphi ,\ldots ,\beta _{2k}^{{}}\varphi ^{p-2}\right\}
\end{equation}

\noindent where $\forall 1\leq k\leq n, \ \beta _{1k}\in \mathbb{R}$ and $\beta_{2k} \geq 0$ with $\beta _{11}\geq \beta _{21}\geq 0,$ that is the union of the $n$ spectra of the real $g$-circulant matrices $S_k$ of order $p$, being $S_1$ nonnegative, find the real $g$-circulant matrix by blocks $A$ such that

\begin{equation*}
\Sigma \left( A\right) =\mathcal{B}.
\end{equation*}
\end{problem}

\noindent The following result will present sufficient conditions to obtain nonnegative $g-$circulant by blocks matrices having a set like $\mathcal{B}$ as its spectrum.
\begin{theorem}
\label{bloqueconstructivo}
Let $p$ be a prime integer and consider $g$ being a cyclic generator of $U(\mathbb{Z}/p\mathbb{Z})$. For $1\leq k \leq n$, consider $$\mathcal{B}_k= \{\beta_{1k}, \beta_{2k}, \beta_{2k}\varphi,\ldots \beta_{2k}\varphi^{p-2}\}$$  where $\forall 1\leq k\leq n, \beta _{1k}\in \mathbb{R}$ and $\beta_{2k} \geq 0$ with $\beta _{11}\geq \beta _{21}\geq 0$. Let $\mathcal{B}=\cup\mathcal{B}_k$. Let $S_1,S_2,\ldots,S_n$ be $g$-circulant real matrices of order $p$, such that $\Sigma(S_k)=\mathcal{B}_k ,\ \forall{1\leq k \leq n}.$ Moreover, suppose that
\begin{eqnarray}
\label{mpc}
\mathcal{B}_{n+2-k}=\overline{\mathcal{B}}_{k}, \quad 2\leq k \leq n
\end{eqnarray}
and that $\mathcal{B}_{1}$ verify the conditions of Theorem
 \ref{suficient}, and that

\begin{eqnarray}
\label{bloquesuficiente}
L_k=\frac{1}{n}\sum\limits_{\ell=1}^{n} S _{\ell}\omega ^{-(k-1)(\ell-1)}\geq 0, \quad \forall{1\leq k \leq n}.
\end{eqnarray}

\noindent Then $\mathcal{B}$ is the spectrum of a nonnegative $g$-circulant matrix by blocks.
\end{theorem}

\begin{proof}
Using Theorem \ref{constructivo}, with the given sets $\mathcal{B}_{k},$ $g$-circulant matrices of orders $p$ can be constructed. Let $S_{k},1\leq k\leq n$ be those matrices. Following the Theorem \ref{main}, the referred sets
are the spectra of the $g$-circulant matrices $S_k\ \forall{1\leq k\leq n}.$
Then, matrices $L_k$ as in (\ref{summatrix}) can be constructed. As the $g$-circulant matrices $S_k$ are permutatively equivalent, in fact they are $g$-circulant, therefore the matrix $L_k$ is also $g$-circulant. It should also be noted that from Remark \ref{remark1} the matrices $L_k$ must be real and therefore the equality
 $$S_{n+2-k}=\overline{S}_{k},  \quad 2\leq k \leq n,$$ must hold. In fact, it comes from the equality in (\ref{mpc}). Moreover, $S_1$ is  nonnegative as the spectrum $\mathcal{B}_{1}$ verifies the condition of Theorem \ref{suficient} that is valid by hypothesis. In the general case, in order that the $g$-circulant matrix by blocks becomes nonnegative, all the matrices  $L_k$ in  (\ref{summatrix}) must be nonnegative. Therefore we obtain the condition  (\ref{bloquesuficiente}).
\end{proof}
\begin{remark}
Taking into account Theorem \ref{suficient}, $\ \forall 1\leq k\leq n$
\[
S_{k}=g-circ\left( \beta _{1k}-\beta _{2k},\beta _{1k}+\left( p-1\right)
\beta _{2k},\beta _{1k}-\beta _{2k},\ldots ,\beta _{1k}-\beta _{2k}\right),
\]
denoting
\[
L_{k}=g-circ\left( h_{1k},h_{2k},\ldots ,h_{pk}\right),
\]%
and using the identity in (\ref{bloquesuficiente}) we obtain
\[\begin{pmatrix}
h_{1k},
h_{2k},
\ldots,
h_{pk}\end{pmatrix}^{T}= \frac{1}{n}M \begin{pmatrix}
1,
\omega ^{-\left( k-1\right) },
\ldots,
\omega ^{-\left( n-1\right) \left( k-1\right) }%
\end{pmatrix}^{T},
\]
where
\begin{align*}
M=
\begin{pmatrix}
\beta _{11}-\beta _{21} & \beta _{12}-\beta _{22} & \beta _{13}-\beta _{23}
& \cdots  &  & \beta _{1n}-\beta _{2n} \\
\beta _{11}+\left( p-1\right) \beta _{21} & \beta _{12}+\left( p-1\right)
\beta _{22} & \cdots  &  &  & \beta _{1n}+\left( p-1\right) \beta _{2n} \\
\beta _{11}-\beta _{21} & \beta _{12}-\beta _{22} & \beta _{13}-\beta _{23}
& \cdots  &  & \beta _{1n}-\beta _{2n} \\
\vdots  & \vdots  & \vdots & \ddots &  & \vdots  \\
\beta _{11}-\beta _{21} & \beta _{12}-\beta _{22} & \beta _{13}-\beta _{23}
& \cdots  &  & \beta _{1n}-\beta _{2n}%
\end{pmatrix}.
\end{align*}
In consequence, the condition
\[
G=\frac{1}{\sqrt{n}}MF^{\ast }\geq 0
\]%
is an equivalent condition to (\ref{bloquesuficiente}).

\end{remark}

\end{document}